\documentclass[11pt,a4paper,reqno]{amsart}
\usepackage{amsmath,amssymb,amsfonts,epsfig,mathrsfs,cite}
\usepackage[T1]{fontenc}
\usepackage{color}
\usepackage{array}
\usepackage{amsthm}
\usepackage{amstext}
\usepackage{graphicx}
\usepackage{setspace}
\usepackage{floatrow} 
\usepackage[title]{appendix}

\usepackage{subfigure}
\makeatletter
\@namedef{subjclassname@2020}{%
  \textup{2020} Mathematics Subject Classification}
\makeatother

\usepackage[margin=2.5cm]{geometry}
\usepackage{color}
\usepackage{enumitem}
\usepackage{amscd,psfrag}
\usepackage{yhmath}
\usepackage[mathscr]{eucal}

\usepackage{comment}

\allowdisplaybreaks[4]

\usepackage{slashed}

\makeatletter
\pdfpageheight\paperheight
\pdfpagewidth\paperwidth

\usepackage{epstopdf}
\usepackage{chngcntr}
\counterwithin{figure}{section}
\usepackage{mathrsfs}

\usepackage{indentfirst}	

\usepackage[normalem]{ulem}
\theoremstyle{plain}
\numberwithin{equation}{section}

\newtheorem{definition}{Definition}[section]
\newtheorem{theorem}[definition]{Theorem}
\newtheorem{notation}[definition]{Notation}
\newtheorem*{theorem*}{Theorem}

\newtheorem{remark}[definition]{Remark}

\newtheorem*{remark*}{Remark}
\newtheorem*{sideremark*}{Side Remark}

\newtheorem*{claim*}{Claim}
\newtheorem*{lemma*}{Lemma}
\newtheorem*{q*}{Question}
\newtheorem{lemma}[definition]{Lemma}
\newtheorem{corollary}[definition]{Corollary}
\newtheorem*{corollary*}{Corollary}

\newtheorem{proposition}[definition]{Proposition}

\newcommand{\R}{\mathbb{R}}

\newcommand{\emb}{\hookrightarrow}

\newcommand{\p}{\partial}

\newcommand{\e}{\varepsilon}

\newcommand{\dd}{{\rm d}}

\newcommand{\E}{{\mathcal{E}}}

\newcommand{\mres}{\mathbin{\vrule height 1.6ex depth 0pt width
0.13ex\vrule height 0.13ex depth 0pt width 1.3ex}}

\def\XXint#1#2#3{{\setbox0=\hbox{$#1{#2#3}{\int}$ }
\vcenter{\hbox{$#2#3$ }}\kern-.6\wd0}}

\title{Characterisations for the depletion of reactant in a one-dimensional dynamic combustion model}

\author{Siran Li}

\address{Siran Li: School of Mathematical Sciences $\&$ CMA-Shanghai, Shanghai Jiao Tong University, No.~6 Science Buildings,
800 Dongchuan Road, Minhang District, Shanghai, China (200240)}

\email{\texttt{siran.li@sjtu.edu.cn}}

\author{Jianing Yang}

\address{Jianing Yang: School of Mathematical Sciences, Shanghai Jiao Tong University, No.~6 Science Buildings,
800 Dongchuan Road, Minhang District, Shanghai, China (200240)}

\email{\texttt{jnyang22@sjtu.edu.cn}}


\date{\today}

\begin{document}

\begin{abstract}
In this paper, a novel observation is made on a one-dimensional compressible Navier--Stokes model for the dynamic combustion of a reacting mixture of $\gamma$-law gases ($\gamma>1$) with discontinuous Arrhenius reaction rate function, on both bounded and unbounded domains. We show that the mass fraction of the reactant (denoted as $Z$) satisfies a weighted gradient estimate $Z_y\slash \sqrt{Z} \in L^\infty_t L^2_y$, provided that at time zero the density is Lipschitz continuous and bounded strictly away from zero and infinity. Consequently, the graph of $Z$ cannot  form cusps or corners  near the points where the reactant in the combustion process is completely depleted at any instant, and the entropy of $Z$ is bounded from above. 
The key ingredient of the proof is a new estimate based on the Fisher information, first exploited by \cite{method1, method2} with applications to PDEs in chemorepulsion and thermoelasticity. Along the way, we also establish a Lipschitz estimate for the density. 
\end{abstract}
\maketitle

\section{Introduction}\label{sec: intro}

\subsection{The model}\label{subsec: model, Euler}
This paper is devoted to the analysis of a one-dimensional (1D) compressible Navier--Stokes model for a reacting mixture in dynamic combustion with discontinuous reacting rate function. This system of partial differential equations (PDEs), written in the following form, was first studied by Chen \cite{chen} in 1992 and has been an active topic of research ever since (\emph{cf}. \cite{cd, cht, dz, dt1, dt2, li, cit1, cit2, cit3, cit4, cit5, cit6, cit7, cit8, cit9, cit10, cit11, cit12, cit13, wang}):
\begin{equation}\label{PDE Euler}
\begin{cases}
\rho_t + (\rho u)_x=0,\\
(\rho u)_t + (\rho u^2 + p)_x = (\mu u_x)_x,\\
(\rho E)_t + [u(\rho E+p)]_x = (\mu uu_x)_x + (\nu \theta_x)_x + (qD\rho Z_x)_x,\\
(\rho Z)_t + (\rho uZ)_x = -K\phi(\theta)\rho Z +(D\rho Z_x)_x.
\end{cases}
\end{equation}
See also  Williams \cite{wil}, Majda \cite{maj}, Glimm \cite{gli},  Collela--Majda--Roytburd \cite{cmr}, and many subsequent developments for mathematical models of dynamic combustion.

In Eq.~\eqref{PDE Euler}: $\rho$, $u$, $\theta$, $E$, and $Z$ denote respectively the density, velocity, temperature, total specific energy, and mass fraction of the reactant. The \emph{positive} constants $\mu$, $\nu$, $D$, and $K$ are the coefficients of bulk viscosity, heat conduction, diffusion, and rate of reactant, respectively. The total specific energy $E$ is given by
\begin{align}\label{E energy}
E = e + \frac{u^2}{2} + qZ,
\end{align}
where $e$ is the specific internal energy and $q>0$ is the difference between the heat of formation of the reactant and the product. For an ideal gas mixture with the same $\gamma$-law ($\gamma>1$), one has the constitutive relations:
\begin{align*}
e = \frac{p}{\rho(\gamma-1)},\qquad \theta = \frac{p}{\rho RM},
\end{align*}
where $R$ is the Boltzmann constant and $M$ is the molecular weight. We abbreviate $$a=RM$$ in the sequel, which is a positive constant.

The reaction rate function $\phi$ in Eq.~\eqref{PDE Euler} obeys \emph{Arrhenius' law}:
\begin{equation}\label{rate fn}
\phi(\theta) = \begin{cases}
\theta^\alpha e^{-A\slash \theta}\qquad \text{ for } \theta> \theta_{\rm ignite} > 0,\\
0\qquad \text{ for } \theta < \theta_{\rm ignite}.
\end{cases}
\end{equation}
Here, $\theta_{\rm ignite}$ is the ignition temperature at which the combustion reaction is initiated and $A>0$ is known as the activation energy. Note that $\phi \in L^\infty(]0,\infty[)$ with a jump discontinuity at  $\theta_{\rm ignite}$.

Consider as in \cite[Eqs.~(1.5)--(1.8)]{chen} the initial-boundary value problem (IBVP) of Eq.~\eqref{PDE Euler} in the space-time domain $[0,T] \times \Omega$ with impermeably insulated boundaries:
\begin{equation}\label{IBVP}
\begin{cases}
(\rho, u, \theta, Z)\big|_{t=0} = \big(\rho_0(x), u_0(x), \theta_0(x), Z_0(x)\big)\qquad \text{ for all } x \in \Omega,\\
\left(u, \theta_x, Z_x\right)\big|_{\p\Omega} = 0 \qquad \text{ for all } t \in [0,T],\\
0< m_0 \leq \rho_0(x),\, \theta_0(x) \leq M_0 < \infty\quad \text{and}\quad 0 \leq Z_0(x) \leq 1 \qquad \text{ for all } x \in \Omega,
\end{cases}
\end{equation}
where $m_0$ and $M_0$ are positive constants.

The initial-boundary data in Eq.~\eqref{IBVP} are subject to the natural compatibility conditions which, in the case of $\gamma$-law ideal gases, read as follows:
\begin{align}\label{compatibility}
\Big(u_0,\, (\theta_0)_x,\, (Z_0)_x,\, a \rho_0\theta_0 - \big[\mu (u_0)_x\big]_x\Big)\Big|_{\p\Omega} = {\bf 0}.
\end{align}

\begin{remark}
The arguments in this paper rely crucially on the \emph{Neumann} boundary condition for $Z$. It is also important to assume that the diffusion constant $D$ is strictly positive.
\end{remark}

We investigate PDEs~\eqref{PDE Euler} and \eqref{IBVP} (subject to Eq.~\eqref{compatibility}) on both bounded domain (without loss of generality, $\Omega = [0,1]$) and unbounded domain $\Omega = \R$.  In the latter case, the boundary conditions --- \emph{i.e.}, the second line in Eq.~\eqref{IBVP} --- are understood as the far-field conditions:
\begin{align}\label{far-field cond}
\lim_{|x| \to \infty}\left(u, \theta_x, Z_x\right) = 0 \qquad \text{ for all } t \in [0,T].
\end{align}


\subsection{Lagrangian coordinates}
As is often the case for 1D compressible fluid dynamics, the Lagrangian formulation of Eq.~\eqref{PDE Euler} proves to be convenient for PDE analysis. We shall reserve the symbol $(t,y)$ for the Lagrangian coordinates, while $(t,x)$ always designates the Eulerian coordinates as in \S\ref{subsec: model, Euler} above. They are related by
\begin{align*}
y = \int_{x(t)}^x \rho(t,s)\,\dd s\qquad &\text{where $x(t)$ is the particle}\\
&\text{trajectory obeying the ODE } x'(t) = u\big(x(t),t\big).
\end{align*}

Here and hereafter, denote by
\begin{align*}
v := \frac{1}{\rho}
\end{align*}
the specific volume of the gas mixture. Also, for simplicity, we take $x(t)=0$ and assume that $\int_{0}^{1}v_0(x)\,\mathrm{d}x=1.$

Eq.~\eqref{PDE Euler} in Lagrangian coordinates $(t,y)$ can be expressed as follows:
\begin{equation}
	\label{eq1}
		\left\{
		\begin{aligned}
			&v_t-u_y=0,\\
&u_t+\left(\frac{a\theta}{v}\right)_y=\left(\frac{\mu u_y}{v}\right)_y,\\
&\left(\theta+\frac{u^2}{2}\right)_t+\left(\frac{au\theta}{v}\right)_y=\left(\frac{\mu u u_y}{v}\right)_y+\left(\frac{\nu \theta_y}{v}\right)_y+qK\phi(\theta)Z,\\
&Z_t+K\phi(\theta)Z=\left(\frac{D}{v^2}Z_y\right)_y.
		\end{aligned}
		\right.
	\end{equation} 
This is a PDE system for $(v,u,\theta,Z)$ in $[0,T] \times \Omega$.

In \cite[\S 2, Theorem~1.1]{chen}, Chen established the equivalence of generalised solutions to Eqs.~\eqref{PDE Euler} and \eqref{eq1}, adapting the arguments in \cite{wag} by Wagner for 1D compressible Navier--Stokes equations. More precisely, it is proved that the transform from Eulerian to Lagrangian coordinates induces a one-to-one correspondence between the generalised solutions to Eqs.~\eqref{PDE Euler} and \eqref{eq1}, for which $\rho$ (or, equivalently, $v$) is essentially bounded away from zero and infinity. Therefore, throughout this paper, we shall work with Eq.~\eqref{eq1} in the Lagrangian coordinates, equipped with the initial-boundary conditions in Eq.~\eqref{IBVP}.

The notion of generalised solution in this work is defined as in \cite[p.611]{chen}.

\begin{definition}
	\label{defsolu}
	Let $\Omega=[0,1]$ or $]-\infty,\infty[$. For any given $T>0$, the quadruplet $$(v,u,\theta,Z): [0,T] \times \Omega \longrightarrow \R^4$$ is a generalised solution to the initial-boundary value problem~\eqref{eq1} $\&$ \eqref{IBVP} if
	
\begin{itemize}
\item
Eq.~\eqref{eq1} holds almost everywhere on $[0,T]\times\Omega$;
\item
the following estimates hold:
\begin{equation*}
		\left\{
		\begin{aligned}
			&(u,\theta-1,Z)\in L^{\infty}\big(0,T;H^{1}(\Omega)\big)\cap L^2\big(0,T;H^2(\Omega)\big),\\
			&(v_t, u_t,\theta_t,Z_t)\in L^2\big(0,T;L^2(\Omega)\big),\\
			& v^{-1}-1 \in L^\infty\big(0,T; H^1(\Omega)\big);
		\end{aligned}
		\right.
	\end{equation*}
	and
	\item
initial conditions in Eq.~\eqref{IBVP} hold in the sense of trace in the above function spaces.
\end{itemize}

\end{definition}
Throughout, for our PDE system in consideration, the initial-boundary conditions~\eqref{IBVP} are always assumed to satisfy the compatibility condition~\eqref{compatibility}.  

\subsection{Known results}
Various results concerning the well-posedness, stability, and asymptotic behaviour of the 1D dynamic combustion model, namely Eqs.~\eqref{eq1} $\&$  \eqref{IBVP}, have been established in the literature. 

In \cite{chen}, Chen proved the existence of generalised solutions on bounded domain $\Omega=[0,1]$, with quantitative bounds in Eq.~(3.33) therein. Moreover, large-time behaviour as follows was proved in \cite[Theorem~5(A)]{chen}:
\begin{align*}
\lim_{t \to \infty} \left\|\left( v-\int_0^1v_0(y)\,\dd y,\, u,\, \theta-\theta^\infty,\, Z-Z^\infty \right)\right\|_{H^1(\Omega)} = 0,
\end{align*}
with the constant asymptotic states $\theta^\infty$, $Z^\infty$ determined via
\begin{equation*}
\theta^\infty + qZ^\infty = \int_0^1\left\{ \theta_0(y) + qZ_0(y)+\frac{u_0^2(y)}{2} \right\}\,\dd y
\end{equation*}
and
\begin{equation*}
\phi\left(\theta^\infty\right)Z^\infty = 0.
\end{equation*}

More refined decay estimates have been obtained by Wang--Wen \cite{cit7} and Wang--Wu \cite{cit4}, among other works.  
The first named author in \cite{li} generalised the aforementioned existence and asymptotic results in \cite{chen} to the unbounded domain case $\Omega = \R$. Other types of boundary conditions have also been considered in \cite{chen, li, cit7, cit4}.

For closely related works, we refer to Feng--Zhang--Zhu \cite{cit9}, Feng--Hong--Zhu \cite{cit8}, Meng \cite{cit5}, Peng \cite{cit10}, Li--Peng \cite{cit1, cit2, cit6}, Xu--Feng \cite{cit12}, and Yin--Zhu \cite{cit3}, etc.~for asymptotic and stability analysis on the combustion model with different initial and/or boundary data, to Chen--Kratka \cite{ck}, Liao--Wang--Zhao \cite{cit13}, and Zhu \cite{cit11}, etc.~for studies on spherically symmetric reacting flows in multi-dimensions, and to Chen--Hoff--Trivisa \cite{cht} for analysis of a two-species model. 

The above list of references is by no means exhaustive.

\subsection{Main theorem and  consequences}

This note aims at providing a novel estimate for the mass fraction of reactant, \emph{i.e.}, the function $Z$ in Eq.~\eqref{PDE Euler} or Eq.~\eqref{eq1}. As a consequence, we show that $Z$ cannot be merely H\"{o}lder or Lipschitz near its nodal set (\emph{a.k.a.} zero locus) $Z^{-1}{\{0\}}$ at any instant, so the reactant cannot be depleted abruptly. This, to the best of the authors' knowledge, is among the very first refined characterisations of solutions to the 1D dynamic combustion model --- Eqs.~\eqref{eq1} $\&$  \eqref{IBVP} --- apart from those obtained via more conventional $L^2$-based energy estimates.

Our main theorem is as follows. 

\begin{theorem}\label{thm: main}
Let $Z$ be the mass fraction of the reactant in a generalised solution to the 1D dynamic combustion model Eqs.~\eqref{eq1} $\&$  \eqref{IBVP} in $[0,T]\times \Omega$, where $T>0$ and $\Omega = [0,1]$ or $\R$. Suppose the initial data satisfy, in addition to the conditions in Eq.~\eqref{IBVP}, that $$(v_0(y),u_0(y),\theta_0(y),Z_0(y))\in \left[H^{1}(\Omega)\right]^4 \quad\text{and}\quad v_0 \in W^{1,\infty}(\Omega).$$

Then we have the estimate
\begin{align*}
\sup_{t \in [0,T]} \left\{\int_\Omega \frac{\left|Z_y\right|^2}{Z} \,\dd y\right\} \leq \Lambda < \infty,
\end{align*}
where $\Lambda$ is a uniform constant as in Notation~\ref{notation} below.
\end{theorem}

\begin{notation}\label{notation}
We reserve the symbol $\Lambda$ for positive constants depending only on the physical parameters $a$, $q$, $m_0$, $M_0$, $D$, $K$, $\mu$, and $\nu$; the lifespan $T$; the $L^\infty$-norm of the discontinuous reaction rate function $\phi$; as well as the natural norm of the initial data $$\left\|\big(v_0, u_0, \theta_0, Z_0\big)\right\|_{[H^1(\Omega) \cap W^{1,\infty}(\Omega)] \times [H^1(\Omega)]^3}.$$ It may change from line to line.
\end{notation}

By Definition~\ref{defsolu}, $Z \in L^\infty\big(0,T; H^1(\Omega)\big)$ and $Z_t \in L^2\big(0,T; L^2(\Omega)\big)$ whenever $(v,u,\theta,Z)$ is a generalised solution.  Theorem~\ref{thm: main} does not provide new information on  $\{Z \geq b_0\}$ for positive definite $b_0>0$. It instead gives refined characterisations for $Z$ near its nodal set $Z^{-1}\{0\}$; \emph{i.e.}, where the reactant of the dynamic combustion is depleted. This explains the title of the paper.

The study of $Z^{-1}\{0\}$ is an interesting topic and appears by now elusive in the literature. In contrast, the other \emph{a priori} non-negative variables $v$ and $\theta$ are shown to be strictly positive. More precisely, it is proved that $$c \leq \theta, \,v\leq C \text{ for strictly positive uniform constants } 0<c\leq C< \infty.$$ See \cite[Lemmata~3, 4, and 6]{chen} and \cite[Theorem~3.1, Lemma~5.1, and Eq.~(5.25)]{li}. In physical terms, the gas mixture in dynamic combustion cannot develop vacua or concentrations, and the temperature remains finite and cannot approach absolute zero. For $Z$, it has been shown by maximum principle arguments that $0 \leq Z \leq 1$ (\emph{cf}. \cite[Lemma~2]{chen} and \cite[Lemma~2.2]{li}. Strictly speaking, these inequalities only hold almost everywhere;  see however Remark~\ref{rem: cont rep} below), but we have had no knowledge about the set of degeneracy  $Z^{-1}\{0\}$.

\begin{remark}\label{rem: cont rep}
By virtue of the Aubin--Lions lemma, $Z$ equals almost everywhere to a function $$\tilde{Z} \in \bigcap_{\gamma \in [0,1[}C^0\big([0,T]; H^{\gamma}(\Omega)\big).$$ In the sequel, we say that $\tilde{Z}$ is a \emph{continuous (in-time) representative} of $Z$ and identify $\tilde{Z}\equiv Z$ without relabelling. 
\end{remark}

Under the above provision, we readily deduce the following from Theorem~\ref{thm: main}.

\begin{corollary}\label{cor}
Near any $y_\star \in Z^{-1}{\{0\}}$ and at \underline{any} time $t \in [0,T]$, the function $Z$ cannot be locally of the form
\begin{align*}
Z(t,y) = C\left|y-y_\star\right|^\beta + {\rm l.o.t.} \qquad \text{for any } C \in \R, \, \beta \in [0,1].
\end{align*}
Here ${\rm l.o.t.}$ denotes a (smooth) perturbation of magnitude $\sim \mathfrak{o}\left(\left|y-y_\star\right|^\beta\right)$ as $y \to y_\star$. The constants $C$ and $\beta$ and the term ${\rm l.o.t.}$ may depend on time.
\end{corollary}

In view of Definition~\ref{defsolu} of the generalised solution, one has $Z \in L^2\big(0,T; H^2(\Omega)\big)$, which implies by Sobolev--Morrey embedding that $Z \in L^2\big(0,T; C^{\frac{3}{2}}(\Omega)\big)$. This of course shows that $Z(t, \bullet) \in  C^{\frac{3}{2}}(\Omega)$ for \underline{almost every} $t \in [0,T]$. Nonetheless, it leaves open the possibility that $Z|_{\mathcal{N} \times \Omega}$ exhibits rather singular behaviours on some Lebesgue null set $\mathcal{N} \subset [0,T]$. ($\mathcal{N}$ does not even need to be discrete here; for example, it may be some fat Cantor sets.) 

In this context, Corollary~\ref{cor} rules out, \underline{at any instant}, the local geometries of the graph of $Z$ near its nodal set as shown in Figure~\ref{figure}, where $y_0, y_1, y_2 \in Z^{-1}\{0\}$. In other words, the graph of $Z$ --- understood as its continuous representative as before --- cannot instantaneously develop cusps or corners at any time on $[0,T]$.

\begin{figure}[H]\caption{Inadmissible shapes of $Z$, the mass fraction of reactant}\label{figure}
	\centering
	\subfigure[]{
		\begin{minipage}{6cm}
		    \centering        
		    \includegraphics[scale=0.7]{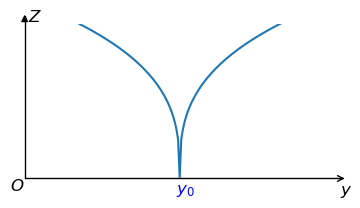}
		\end{minipage}
    }
    \subfigure[]{
    	\begin{minipage}{6cm}
    		\centering        
    		\includegraphics[scale=0.7]{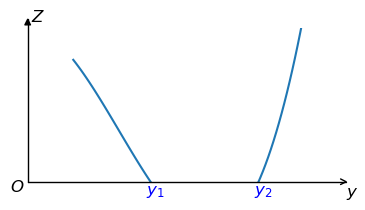}
    	\end{minipage}
    }

\end{figure}

\begin{proof}[Proof of Corollary~\ref{cor}]
The arguments below are local --- \emph{i.e.}, only pertaining to arbitrarily small neighbourhoods of $Z^{-1}{\{0\}}$. So we may treat the cases $\Omega = [0,1]$ and $\R$ at one stroke. 

If $\beta \in [-0,1[$, then $Z=\left|y-y_\star\right|^\beta$ (without loss generality, assuming $C=1$ and ${\rm l.o.t.}=0$ here) leads to ${Z_y^2}\slash{Z}=\beta^2\left|y-y_\star\right|^{\beta-2}$, which is not in $L^1$ locally about $y_\star$.

If $\beta =1$ and $Z=\left|y-y_\star\right|$, then $Z_y^2 = 1$ on $\Omega \setminus \{y_\star\}$, so almost everywhere $Z_y^2 \slash Z = \left|y-y_\star\right|^{-1}$, which again is not in $L^1_{\rm loc}(\Omega)$.  \end{proof}

Another consequence of Theorem~\ref{thm: main} is the boundedness of \emph{entropy}:
\begin{corollary}\label{cor: entropy}
Let $Z$ be the mass fraction of the reactant in a generalised solution to the 1D dynamic combustion model Eqs.~\eqref{eq1} $\&$  \eqref{IBVP} in $[0,T]\times \Omega$, where $T>0$ and $\Omega = [0,1]$. Then, the entropy of $Z$ with respect to the Lebesgue measure on $[0,1]$, defined by
\begin{align*}
{\rm Ent}_{\mathcal{L}^1\mres [0,1]}(Z) := \int_0^1 Z\log Z\,\dd y - \left(\int_0^1Z\,\dd y\right) \log\left(\int_0^1 Z\,\dd y\right),
\end{align*}
is bounded from above on $[0,T]$ by a uniform constant, depending on the same parameters as $\Lambda$ in Theorem~\ref{thm: main}. 
\end{corollary}

Note that ${\rm Ent}_{\mathcal{L}^1\mres [0,1]}(Z) \geq 0$ by convexity; here we understand $0 \log 0 = 0$. For the unbounded domain case $\Omega = \R$, one may replace the entropy by the following variant:
\begin{align*}
{\rm Ent}_{\gamma}(Z) := \int_0^1 Z\log Z\,\dd \gamma(y) - \left(\int_0^1Z\,\dd \gamma(y)\right) \log\left(\int_0^1 Z\,\dd \gamma(y)\right),
\end{align*}
where $\gamma$ is the Gaussian probability measure $\dd \gamma(y) = (2\pi)^{-1/2} \exp\{-y^2/2\}\,\dd y$ on $\R$.

\begin{proof}[Proof of Corollary~\ref{cor: entropy}]
It follows immediately from Theorem~\ref{thm: main} and the Logarithmic Sobolev inequality $${\rm Ent}_{\mathcal{L}^1\mres [0,1]}(Z) \leq c_0 \int_0^1 \frac{Z_y^2}{Z}\,\dd y.$$ Here $c_0$ is a universal constant determined, for example, from the curvature-dimension condition of $\mathcal{L}^1\mres [0,1]$ on $\R$. See Bakry--Gentil--Ledoux \cite[\S 5.1]{bgl}.  
\end{proof} 

The key ingredient of the proof for our main Theorem~\ref{thm: main} is an inequality based on Fisher information. See Lemma~\ref{diffineq}. This inequality was proved in the recent preprint \cite{method1} on a 3D chemorepulsion system by Cie\'{s}lak--Fuest--Hajduk--Sier\.{z}\c{e}ga, and later applied to 1D thermoelasticity PDE models by Bies--Cie\'{s}lak  in \cite{method2}. In addition, careful uniform \emph{a priori} estimates on the generalised solution $(v,u, \theta, Z)$ are performed to deal with the discontinuity of the reaction rate function $\phi$.

\subsection{Organisation}
The remaining sections of the paper are organised as follows. 

In \S\ref{sec: energy}, we collect several uniform \emph{a priori} bounds for the generalised solution, which are crucial for later developments. Then, in \S\ref{sec: Fisher} we introduce the Fisher information-based inequality. This is the key tool for this paper. \S\S\ref{sec: gradient estimate on bd} $\&$ \ref{sec: gradient estimate on ubd} are devoted to the proof of the main Theorem~\ref{thm: main}, in the cases of $\Omega = [0,1]$ and $\Omega = \R$, respectively.

\section{Preliminary estimates}\label{sec: energy}

In this section, we collect from the existing literature (especially Chen \cite{chen}; Jiang \cite{j1}) several uniform estimates for the dynamic combustion model Eqs.~\eqref{eq1} $\&$ \eqref{IBVP}, which are used to establish  existence results --- summarised as Proposition~\ref{GST} below --- for generalised solutions. We also fix some notations that will be used throughout this work.

\begin{proposition}
	\label{GST}
	Suppose that the initial data 
	$$(v_0(y),u_0(y),\theta_0(y),Z_0(y))\in \left[H^{1}(\Omega)\right]^4$$
satisfy Eqs.~\eqref{IBVP} $\&$ \eqref{compatibility}. There exists a generalised solution $(v,u,\theta,Z)$ to Eqs.~\eqref{eq1} $\&$ \eqref{IBVP}. 
\end{proposition}

\begin{remark}\label{rem: lip}
For the existence of generalised solutions, one does not need to assume the stronger condition $v_0 \in W^{1,\infty}(\Omega)$  as in the main Theorem~\ref{thm: main}. $H^1$-regularity will suffice.
\end{remark}

\subsection{Regularised system}\label{subsec: reg}
The proof of Proposition~\ref{GST} in \cite{chen, li} relies on the $C^1$-regularisation of the discontinuous reaction rate function $\phi(\bullet)$ defined in Eq.~\eqref{rate fn}, which shall be described as follows for convenience of the readers.

Define the \emph{nondecreasing} function
\begin{align}\label{regularised Phi}
&\Phi_{\epsilon}(\theta)=\left\{
	\begin{aligned}
		&\theta^{\alpha}e^{-A/\theta}\qquad \text{ for }  \theta\geq \theta_{\rm ignite}+\epsilon,\\
		&0\qquad \text{ for }  \theta\leq \theta_{\rm ignite}-\epsilon,
	\end{aligned}
	\right. \qquad\text{and}\qquad \Vert \Phi_{\epsilon}\Vert_{C^1}\leq \frac{1}{\epsilon}.
\end{align} 
One obtains uniform estimates independently of the regularisation parameter $\e>0$, and then passes to limits in suitable Sobolev spaces by sending $\e \searrow 0$. 

The aforementioned uniform estimates are naturally derived in the (parabolic) H\"{o}lder framework. We adopt the following notations from \cite{chen}. 
Let $Q_T=[0,T]\times \Omega$ and set
$$\left\{\begin{aligned}
	&B^{\sigma/2,\sigma}(Q_T)=\left\{u\in C^0(Q_T):\sup_{Q_T}|u|+\sup_{t\neq s,x\neq y}\frac{|u(t,x)-u(s,y)|}{|t-s|^{\sigma/2}+|x-y|^{\sigma}}<\infty\right\}, \\
    &B^{1+\sigma}(Q_T)=\left\{u\in B^{\sigma/2,\sigma}(Q_T):u_t,u_y\in B^{\sigma/2,\sigma}(Q_T)\right\},\\
    &B^{2+\sigma}(Q_T)=\left\{u\in B^{\sigma/2,\sigma}(Q_T):u_t,u_y,u_{yy}\in B^{\sigma/2,\sigma}(Q_T)\right\}.
\end{aligned}\right. $$
We denote the norms $\Vert\bullet\Vert_{1+\sigma}:=\Vert\bullet\Vert_{B^{1+\sigma}}$ and $\Vert\bullet\Vert_{2+\sigma}:=\Vert\bullet\Vert_{B^{2+\sigma}}$, where  $T>0$ is suppressed.

In the case $\Omega = [0,1]$, it is established in \cite{chen} by linearisation and classical Schauder estimates that, for regular initial data, there exists a unique classical solution in parabolic H\"{o}lder spaces. The analogue on $\Omega= \R$ is obtained in \cite{li}. The crucial ingredients of the arguments therein are adapted from Jiang \cite{j1, j2} (which, in turn, rely on the localisation lemma by Kazhikhov--Shelukhin \cite{ks, ks2}) and Li--Liang \cite{ll16}.

\begin{proposition}[Theorem~3 in \cite{chen} and Theorem~1.2 in \cite{li}]
	\label{th1}
Let $\Omega = [0,1]$ or $\R$. Suppose that for some $\alpha \in ]0,1]$ the initial data $(v_0, u_0, \theta_0, Z_0)$ satisfy the conditions in Eq.~\eqref{IBVP} and
	\begin{equation}
		\label{inire}
		v_0\in C^{1,\alpha}(\Omega),\quad (u_0,\theta_0,Z_0)\in C^{2,\alpha}(\Omega),\quad
		\Vert \left(u_0,\theta_0,Z_0\right)\Vert_{H^1(\Omega)}\leq M_1.
	\end{equation}
Consider the system~\eqref{eq1} $\&$ \eqref{IBVP}, in which the discontinuous reaction rate function $\phi$ is replaced by its $C^1$-regularisation $\Phi_\e$ defined in Eq.~\eqref{regularised Phi}. Then for aribitrary $T>0$, there exists a unique classical solution $$\left(v^\e,u^\e,\theta^\e,Z^\e\right)\in B^{1+\sigma}(Q_T)\times \big[B^{2+\sigma}(Q_T)\big]^3$$ for some exponent $\sigma \in ]0,1]$ depending on $\alpha$ and $M_1$.

In addition, the following estimates hold for all $(t,y) \in Q_T$:
	\begin{equation*}
		\left\{
		\begin{aligned}
			&C_0^{-1}\leq v^\e\leq C_0,\quad 0<C^{-1}\leq \theta^\e\leq C_0,\quad |u^\e|\leq C_0,\quad 0\leq Z^\e\leq 1,\\
			&\left\| \left(v^\e_t,v^\e_y\right)\right\|^2(t)+\int_{0}^{t}\left\| \left(v^\e_t,v^\e_y\right)\right\|^2(\tau)\,\mathrm{d}\tau\leq C_0,\\
			&\left\|\left( u^\e_y,\theta^\e_y,Z^\e_y\right)\right\|^2(t)+\int_{0}^{t}\left\|\left( u^\e_{yy},\theta^\e_{yy},Z^\e_{yy},u^\e_y,\theta^\e_y,Z^\e_y,u^\e_t,\theta^\e_t,Z^\e_t\right)\right\|^2(\tau)\,\mathrm{d}\tau\leq C_0,\\
			&\Vert v^\e\Vert_{1+\sigma} + \left\|\left( u^\e,\theta^\e,Z^\e\right)\right\|_{2+\sigma}\leq C.
		\end{aligned}\right.
	\end{equation*} 
Here $C_0$ depends only on $M_1$, while $C$ depends only on $M_1$, $\epsilon$, and $T$.
\end{proposition}

In particular, note that Proposition~\ref{th1} together with the monotonicity of $\Phi_{\epsilon}$ yields that $\Phi_{\epsilon}(\theta)$ is uniformly bounded from above independently of $\epsilon$.

Next, let us document several conserved quantities of the system~\eqref{eq1} $\&$ \eqref{IBVP}, which both play a crucial role in the proof of Proposition~\ref{th1} and have enormous physical significance. See \cite[Lemma~1]{chen} and \cite[Propositions~2.3 $\&$ 2.4]{li} for proofs.

\begin{lemma}\label{lem: conserved quantities}
Let $\Omega = [0,1]$ or $\R$; $T>0$. For any $t \in [0,T]$, the classical solution $\left(v^\e,u^\e,\theta^\e,Z^\e\right)$ on $Q_T$ (whose existence is guaranteed by Proposition~\ref{th1}) satisfies the identities:  
	\begin{eqnarray*}
	&& 	\int_\Omega v(t,y)\,\mathrm{d}y=\int_\Omega v_0(y)\,\mathrm{d}y,\\
	&& \int_\Omega Z(t,y)\,\mathrm{d}y+\int_{0}^{t}\int_\Omega K\phi(\theta)Z\,\mathrm{d}y\,\mathrm{d}\tau =\int_{0}^{1}Z_0(y)\,\mathrm{d}y,\\
	&& \int_\Omega \left[a(v-1-\log v)+\frac{u^2}{2}+(\theta-1-\log \theta)\right]\,\mathrm{d}y+\int_{0}^{t}\int_\Omega \left[\frac{\mu u_y^2}{v\theta}+\frac{\nu \theta_y^2}{v\theta^2}-q\frac{\theta-1}{\theta}K\phi(\theta)Z\right]\,\mathrm{d}y\,\mathrm{d}\tau\\
&&\qquad	= \int_{\Omega}\left[a(v_0-1-\log v_0)+\frac{u_0^2}{2}+(\theta_0-1-\log \theta_0)\right]\,\mathrm{d}y\\
&&\qquad =: \E_0<\infty.
	\end{eqnarray*}
\end{lemma}

\subsection{Uniform bounds on $v_y$}

For the proof of Theorem~\ref{thm: main}, one crucial estimate we need is the spacetime $L^\infty$-bound for $v_y$, the first spatial derivative  of $v$. Of course, such estimate is impossible unless it holds at $t=0$. This motivates the assumption
\begin{align*}
v_0 \in H^1(\Omega) \cap W^{1,\infty}(\Omega)
\end{align*}
in Theorem~\ref{thm: main}. As commented in Remark~\ref{rem: lip}, the Lipschitz assumption is in fact not needed to establish the existence of generalised solutions (see Definition~\ref{defsolu}).

The following result shows that the Lipschitzness of $v$ is propagated through time. 
\begin{theorem}\label{thm: estimate on vy}
Let $T>0$ and $\Omega = [0,1]$. Suppose that the initial data $(v_0, u_0, \theta_0, Z_0)$ satisfy the conditions in Theorem~\ref{thm: main}. There is a finite constant $\Lambda$ as in  Notation~\ref{notation} such that
\begin{align*}
\left\|v_y\right\|_{L^\infty(Q_T)} \leq \Lambda.
\end{align*}
\end{theorem}

Proposition~\ref{GST} and Theorem~\ref{thm: estimate on vy} together imply that $$v \in L^\infty\big(0,T;H^1(\Omega) \cap W^{1,\infty}(\Omega)\big).$$ 

The essential tool needed for the proof, which has already appeared in \cite{chen}, is an integral representation formula for $v$ (Eq.~\eqref{rep form} below) discovered by Kazhikhov \cite{ks}. We shall only sketch the arguments here, with emphasis on  the dependency of constants. Special care is taken to show that the Lipschitz estimate of $v$ requires only the Lipschitz bounds for the initial specific volume $v_0$, but \emph{not} for $u_0$, $Z_0$, $\theta_0$; the $H^1$-bounds for the later three variables would suffice, as in Definition~\ref{defsolu} of the generalised solutions.

Meanwhile, to avoid undesired complications that may blur the main ideas of the proof, we focus on the case of bounded domain $\Omega = [0,1]$. Theorem~\ref{thm: estimate on vy} for $\Omega=\R$ is also valid, but we postpone its proof to \S\ref{sec: gradient estimate on ubd} towards the end of the paper.

\begin{proof}[Proof of Theorem~\ref{thm: estimate on vy}] We first introduce some notations, most of which are taken from \cite{chen}. Denote by $\alpha_0$, $\beta_0$ the roots $y_0$ of the transcendental equation
\begin{align*}
y_0-1-\log y_0 = \E_1,
\end{align*}
where
\begin{align*}
\E_1:=\frac{1}{\min\left\{1,a\right\}}\left\{\E_0+q\int_{0}^{1}Z_0(y)\,\mathrm{d}y\right\}.
\end{align*}
Clearly $\alpha_0$, $\beta_0$ are positive numbers (as $q>0$ for combustion reaction).

Throughout this proof, we consider the \emph{regularised} system as in \S\ref{subsec: reg}, with $\Phi_\e$ (see Eq.~\eqref{regularised Phi}) in lieu of $\phi$. For notational convenience, we shall drop the superscipt ${}^\e$ in the remaining parts of the proof, namely that $$(v_0, u_0, \theta_0, Z_0) \equiv (v^\e_0, u^\e_0, \theta^\e_0, Z^\e_0)\quad \text{ and } \quad (v,u,\theta,Z)\equiv (v^\e, u^\e, \theta^\e, Z^\e).$$ The process for passage of limits is the same as in Chen \cite{chen}, so all we need to complete the proof is to obtain an $L^\infty$-bound for $v_y$ \emph{uniform in $\e$ and independent of higher-order norms}. See the following paragraph for details.

Under this provision, we may infer from Notation~\ref{notation} that the constants $\Lambda$ are allowed to depend only on the physical parameters $a$, $q$, $D$, $K$, $\mu$, and $\nu$; the constants $m_0$, $M_0$, $\alpha_0$, $\beta_0$, $\E_0$, $\|\phi\|_{L^\infty}$, $T$; as well as the natural norms $
\left\|\big(v_0, u_0, \theta_0, Z_0\big)\right\|_{[H^1(\Omega) \cap W^{1,\infty}(\Omega)] \times [H^1(\Omega)]^3}$ of the initial data. It, in particular,  is independent of the regularisation parameter $\e$ and the H\"{o}lder norms $\Vert v_0\Vert_{1+\sigma}$, $\Vert (u_0, \theta_0, Z_0)\Vert_{2+\sigma}$. As a consequence, $\Lambda$ is persistent under the limit $\e \searrow 0$ passing from the classical solution to the regularised system to the generalised solution to the original system~\eqref{IBVP} $\&$ \eqref{eq1}. Especially, the remark ensuing Proposition~\ref{th1} explains why the dependence on $\|\phi\|_{L^\infty}$ works here.  


Note that the Lemma \ref{lem: conserved quantities} gives us that
\begin{eqnarray*}
&&\int_{0}^{1}a(v-1-\log v)\,\mathrm{d}y\leq \E_0+q\int_{0}^{1}Z_0(y)\,\mathrm{d}y,\\
&&\int_{0}^{1}(\theta-1-\log \theta)\,\mathrm{d}y\leq \E_0+q\int_{0}^{1}Z_0(y)\,\mathrm{d}y.
\end{eqnarray*}
Thus, there exists at least one point $y_\star(t)\in [0,1]$ such that
$$\alpha_0\leq v\big(t,y_\star(t)\big),\,\theta\big(t,y_\star(t)\big)\leq \beta_0.$$

In what follows, we shall adapt the arguments in Kazhikhov \cite{ks} to derive an explicit representation formula for  $v=v(t,y)$. Let us rewrite the first equality in Eq.~\eqref{eq1} as $$\frac{\partial}{\partial t}\log v(t,y)=\frac{u_y(t,y)}{v(t,y)}$$ and integrate it over time. We thus obtain
	\begin{equation}
		\begin{aligned}
			 \frac{\partial}{\partial y}\left[\log v(t,y)-\int_{0}^{t}\frac{a\theta(t,y)}{\mu v(t,y)}\,\mathrm{d}\tau\right]
			 =\frac{\mathrm{d}}{\mathrm{d}y}\log v_0(y)+\frac{1}{\mu}[u(t,y)-u_0(y)].
			\end{aligned}	
		\end{equation}
Integrating the above equality from $y_\star(t)$ to $y$ and taking exponential, we deduce that
	\begin{equation}
		\label{rep form}
		A(t)v(t,y)=\frac{1+ \frac{a}{\mu} \int_{0}^{t}\theta(\tau,y)A(\tau)B(\tau,y)\,\mathrm{d}\tau}{B(t,y)},
	\end{equation}
    where
	$$A(t)=v_0\big(y_\star(t)\big)\exp\left\{\frac{a}{\mu}\int_{0}^{t}\frac{\theta}{v}\big(\tau,y_\star(t)\big)\,\mathrm{d}\tau\right\}$$ and 
	$$B(t,y)=\frac{1}{v_0(y)v\big(t,y_\star(t)\big)}\exp\left\{\frac{1}{\mu}\int_{y_\star(t)}^y [u_0(\xi)-u(t,\xi)]\,\mathrm{d}\xi\right\}.$$

	Note that for any $t$, the Cauchy--Schwarz inequality gives us  
	$$\left|\int_{y_\star(t)}^y u(t,\xi) \,\mathrm{d}\xi\right|\leq \Vert u(t,\bullet)\Vert_{L^2(\Omega)}\leq \sqrt{2\E_1}.$$
	This together with Eq.~\eqref{inire} implies 
	\begin{equation}
		\label{bB}
		0<\Lambda^{-1}\leq B(t,y)\leq \Lambda<\infty\qquad \text{for all } (t,y) \in Q_T,
	\end{equation}
where $\Lambda$ is the constant in Notation~\ref{notation}. Meanwhile, integrating Eq.~\eqref{rep form} over the spatial variable $y\in \Omega = (0,1)$ and using the first identity in Lemma~\ref{lem: conserved quantities}, we arrive at the integral inequality
	\begin{equation*}
0<\Lambda^{-1} \leq A(t) \leq \Lambda+\Lambda\int_{0}^{t}A(\tau)\,\mathrm{d}\tau.
		\end{equation*}
Hence, Gr\"{o}nwall's inequality implies $A(t)\leq \Lambda.$ The constant $\Lambda$ may depend on the lifespan $T$.

	We may now deduce from the representation formula~\eqref{rep form} an explicit expression for $v_y(t,y)$. Indeed, by chain rule and the fundamental theorem of calculus, we compute that
	\begin{equation}
		\label{vy}
		\begin{aligned}
			v_y(t,y)&=\frac{1+\frac{a}{\mu}\int_{0}^{t}\theta(\tau,y)A(\tau)B(\tau,y)\,\mathrm{d}\tau}{A(t)}\left[-\frac{B_y(t,y)}{B^2(t,y)}\right]\\
			&\qquad+\frac{1}{A(t)B(t,y)}{\int_{0}^{t}\frac{a}{\mu}A(\tau)\big[\theta_y(\tau,y)B(\tau,y)+\theta(\tau,y)B_y(\tau,y)\big]\,\mathrm{d}\tau},
		\end{aligned}
	\end{equation}
	and
	\begin{equation*}
			B_y(t,y)=B(t,y)\left[-\frac{v_{0y}(y)}{v_0(y)}+\frac{1}{\mu}\big(u_0(y)-u(t,y)\big)\right].
	\end{equation*}
Here $v_{0y}=\frac{\p v_0}{\p y}$. Thus, by the $L^\infty$-bound for $u$ (see  Proposition~\ref{th1}), the assumptions on the initial data in the statement of Theorem~\ref{thm: estimate on vy} (in particular, the strictly positive lower bound for $v_0$ and its Lipschitz upper bound), and estimate~\eqref{bB} above, we deduce that
	\begin{equation}\label{new2}
		|B_y(t,y)|\leq \Lambda\qquad \text{for all } (t,y) \in Q_T.
	\end{equation}
It is crucial that the above expression for $B_y$ does not involve any derivative of $u$ or $u_0$.
	 
Also observe the following simple bound: for any $t\in [0,T]$ and $y \in \Omega$, 
\begin{align}\label{new1}
\left|\int_{0}^{t} \theta_y(\tau,y)\mathrm{d}\tau \right|&= \left|\int_{0}^{t}\int_{0}^{y} \theta_{yy}(\tau,\xi)\,\mathrm{d}\xi\,\mathrm{d}\tau\right| \nonumber\\
& \leq \int_{0}^{t}\left\| \theta_{yy}(\tau,\bullet)\right\|_{L^2(\Omega)}\,\mathrm{d}\tau\leq \Lambda.
\end{align}
We used here the fundamental theorem of calculus, the boundary condition for $\theta$ at $y=0$, the Cauchy--Schwarz inequality, as well as  Proposition~\ref{th1}.

Therefore, putting together the bounds in Eqs.~\eqref{bB}, \eqref{new2}, $\&$ \eqref{new1} and substituting them into Eq.~\eqref{vy}, we readily conclude that $$|v_y(t,y)|\leq \Lambda \qquad \text{for all } (t,y) \in Q_T.$$ This completes the proof.  \end{proof}

Note that the Lipschitzness of $v$ is equivalent to that of $\rho = v^{-1}$. This is because
\begin{align*}
\frac{1}{v(t,y_1)} - \frac{1}{v(t,y_2)} = \frac{v(t,y_2)-v(t,y_1)}{v(t,y_1)v(t,y_2)}\qquad \text{for any $t \in [0,T]$ and $y_1, y_2 \in \Omega$},
\end{align*}
and $v$ is strictly bounded below by a positive constant (see Proposition~\ref{th1}).

\section{An inequality based on Fisher information}\label{sec: Fisher}

Now let us introduce the key inequality in this work. It was first discovered in the recent preprint \cite{method1} by Cie\'{s}lak--Fuest--Hajduk--Sier\.{z}\c{e}ga, which relates the Hessian of the square-root of a positive function to the dissipation of Fisher information along heat flow.

\begin{lemma}
	\label{diffineq}
Let 	$- \infty \leq a < b \leq \infty$, and let $\Psi$ be a positive $C^2$-function on $]a,b[$ such that $\Psi_x(a)=\Psi_x(b)=0$. Then it holds that
	\begin{equation*}
	   \int_{a}^{b}\left[(\Psi^{1/2})_{xx}\right]^2\,\mathrm{d}x\leq \frac{13}{8}\int_{a}^{b}\Psi\left[(\log \Psi)_{xx}\right]^2\,\mathrm{d}x.
	\end{equation*}
\end{lemma}

When $a = -\infty$, the condition $\Psi_x(a)=0$ is understood as $\lim_{y \searrow -\infty}\Psi_x(y)=0$; similarly for the case $b=\infty$. For convenience of the readers, we shall sketch a proof below. We essentially follow \cite{method1} and do not claim any originality for our arguments. In \cite{method1} the inequality is established only for finite $a$ and $b$, but the same proof carries over to the general case without much alterations. Of course, if either  $a$ or $b$ is infinite, the inequality is effective only when $\int_{a}^{b}\Psi\left[(\log \Psi)_{xx}\right]^2\,\mathrm{d}x<\infty$. 

\begin{proof}[Proof of Lemma~\ref{diffineq}]
For $0 <\Psi \in C^2(]a,b[)$, one has 
\begin{align*}
& \left[(\Psi^{1/2})_{xx}\right]^2=\left(\frac{\Psi_{xx}}{2\Psi^{1/2}}-\frac{\Psi_x^2}{4\Psi^{3/2}}\right)^2=\frac{1}{4}\left(\frac{\Psi_{xx}}{\Psi^{1/2}}-\frac{1}{2}\frac{\Psi_x^2}{\Psi^{3/2}}\right)^2,\\
&\Psi\left[(\log \Psi)_{xx}\right]^2=\Psi\left(\frac{\Psi_{xx}}{\Psi}-\frac{\Psi_x^2}{\Psi^{2}}\right)^2=\left(\frac{\Psi_{xx}}{\Psi^{1/2}}-\frac{\Psi_x^2}{\Psi^{3/2}}\right)^2.
\end{align*}
Completing the square and using the inequality $(a+b)^2\leq 2(a^2+b^2)$, we obtain that
\begin{equation}
			\begin{aligned}
				   \left[(\Psi^{1/2})_{xx}\right]^2&=\frac{1}{4}\left(\frac{\Psi_{xx}}{\Psi^{1/2}}-\frac{\Psi_x^2}{\Psi^{3/2}}+\frac{1}{2}\frac{\Psi_x^2}{\Psi^{3/2}}\right)^2\\
				   &\leq\frac{1}{4}\left[2\left(\frac{\Psi_{xx}}{\Psi^{1/2}}-\frac{\Psi_x^2}{\Psi^{3/2}} \right)^2+\frac{1}{2}\left(\frac{\Psi_x^2}{\Psi^{3/2}}\right)^2\right]\\
				   &=\frac{1}{2}\Psi\left[(\log \Psi)_{xx}\right]^2+\frac{1}{8}\left(\frac{\Psi_x^2}{\Psi^{3/2}}\right)^2.
				\end{aligned}
				\nonumber
		\end{equation}

To conclude, one invokes the following inequality attribute  in \cite{method1} to Winkler \cite{winkler}. In full generality, it reads:
\begin{lemma*}[Bernis-type inequality]
	\label{Btin}
Let $h\in C^{1}(\mathbb{R}_{+})$ be a positive function, and let $\Omega\subset \R^N$ be a smooth Euclidean domain. Set $\Theta(s):=\int_{1}^{s}\frac{\mathrm{d}\sigma}{h(\sigma)}$ for $s>0$. Then, for any positive function $\varphi\in C^{2}\big(\bar{\Omega}\big)$ with $\frac{\partial \varphi}{\partial \nu}=0$ on $\partial \Omega$, it holds that
	$$\int_{\Omega}\frac{h^{\prime}(\varphi)}{h^{3}(\varphi)}\left|\nabla \varphi\right|^{4}\,\dd x\leq (2+\sqrt{N})^2\int_{\Omega}\frac{h(\varphi)}{h^{\prime}(\varphi)}\left|D^2\Theta(\varphi)\right|^{2}\,\dd x.$$
\end{lemma*}

Applying the above inequality to $\Psi = \varphi$, $h={\rm Id}$, $N=1$, and integrating the previous pointwise inequality over space, we arrive at 
\begin{align*}
\int_{0}^{1}\left[(\Psi^{1/2})_{xx}\right]^2\,\mathrm{d}x \leq \frac{13}{8}\int_{0}^{1}\Psi\left[(\log \Psi)_{xx}\right]^2\,\mathrm{d}x,
\end{align*}
which is what we want to prove.  \end{proof}

In passing, we note that the reverse inequality is immediate:
\begin{align*}
\int_{a}^{b}\Psi\left[(\log \Psi)_{xx}\right]^2\,\mathrm{d}x \leq 4\int_{a}^{b}\left[(\Psi^{1/2})_{xx}\right]^2\,\mathrm{d}x.
\end{align*}
In fact, one has the pointwise identity
\begin{align*}
\left(\Psi^{1/2}\right)_{xx} - \frac{1}{2} \Psi^{1/2} \left(\log \Psi\right)_{xx} = \frac{1}{4} \frac{\left(\Psi_x\right)^2}{\Psi^{3/2}}
\end{align*}
for $\Psi$ as in Lemma~\ref{diffineq}.

\section{Weighted gradient estimate for reactant fraction on bounded domain}\label{sec: gradient estimate on bd}

In this section, we prove Theorem~\ref{thm: main} for $\Omega = [0,1]$. 

\begin{proof}[Proof of Theorem~\ref{thm: main} for bounded domain]

As in the proof for Theorem~\ref{thm: estimate on vy}, let us  consider the regularised system --- \emph{i.e.}, with $\Phi \equiv \Phi_\e$ in place of $\phi$ --- and we systematically drop the superscripts ${}^\e$ in our variables.

It suffices to seek a uniform-in-$\e$ bound for $\frac{Z_y}{\sqrt{Z}}$ in $L^\infty_t L^2_y$ by $\Lambda$ (see  Notation~\ref{notation}). Then the theorem follows from a standard passage of limit process, which can be directly adapted from \cite{chen}. 
To avoid potential dangers of dividing by zero, we introduce the new variable
\begin{equation}\label{X, new var}
X := Z+\delta \equiv Z^{\epsilon}+\delta\qquad \text{ for } \e, \delta>0,
\end{equation}
and look for uniform estimates independent of both the regularisation parameter $\e$ and the ``shift'' parameter $\delta$. Say $\delta < 100^{-1}$.

Note that $X$ satisfies the PDE
\begin{equation}
	\label{X PDE}\left\{
	\begin{aligned}
		&X_t+K\Phi(\theta)X-K\Phi(\theta)\delta=\left(\frac{D}{v^2}X_y\right)_y\qquad\text{in } Q_T,\\
		&X_0(y)=Z_0(y)+\delta \qquad\text{at } t=0,\\
		&X_y=0 \qquad \text{on } \p\Omega = \{0,1\}. 
	\end{aligned}\right.
\end{equation}
Here $X_0$ is the initial data $X|_{t=0}$. Dividing $X^{1/2}$ on both sides of Eq.~\eqref{X PDE}, we obtain the evolution equation for $X^{1/2}$:
\begin{equation}
	2\frac{\p}{\p t}X^{1/2}+K\Phi(\theta)X^{1/2}-\frac{K\Phi(\theta)\delta}{X^{1/2}}=\frac{(\frac{D}{v^{2}}X_y)_y}{X^{1/2}}.
	\nonumber
\end{equation}
As $X$ is strictly positive and has the same regularity as $Z$, the above PDE is understood in the classical sense.

We deduce by integrating over $\Omega =[0,1]$ that
\begin{align*}
&\frac{\mathrm{d}}{\mathrm{d}t}\int_{0}^{1}\frac{X_y^2}{X}\,\mathrm{d}y\\
&\quad= 4\frac{\mathrm{d}}{\mathrm{d}t}\int_{0}^{1}(X^{1/2})_y^2\,\mathrm{d}y\\
&\quad=8\int_{0}^{1}(X^{1/2})_y(X^{1/2})_{ty}\,\mathrm{d}y\\
&\quad=4\int_{0}^{1}(X^{1/2})_y\left[\left(\frac{D}{v^{2}}X_y\right)_y \frac{1}{X^{1/2}}-K\Phi(\theta)X^{1/2}\right]_{y} \,\mathrm{d}y\\
&\quad=4\int_{0}^{1}(X^{1/2})_y\left[2\frac{D}{v^2}(X^{1/2})_{yy}+\frac{1}{2}\frac{D}{v^2}\frac{X_y^2}{X^{3/2}}+\left(\frac{D}{v^2}\right)_y\frac{X_y}{X^{1/2}}+K\Phi(\theta)\left(\frac{\delta}{X^{1/2}}-X^{1/2}\right)\right]_{y}\,\mathrm{d}y.
\end{align*}
The penultimate equality follows from the PDE for $X^{1/2}$.

Let us write
\begin{align}\label{I1, I2, I3}
\frac{\mathrm{d}}{\mathrm{d}t}\int_{0}^{1}\frac{X_y^2}{X}\,\mathrm{d}y = I_1+I_2+I_3
\end{align} 
where
\begin{eqnarray*}
&&I_1:=4\int_{0}^{1}(X^{1/2})_y\left[2\frac{D}{v^2}(X^{1/2})_{yy}+\frac{1}{2}\frac{D}{v^2}\frac{X_y^2}{X^{3/2}}\right]_{y}\,\mathrm{d}y,\\
&&I_2:=4\int_{0}^{1}(X^{1/2})_y\left[\left(\frac{D}{v^2}\right)_y\frac{X_y}{X^{1/2}}\right]_{y}\,\mathrm{d}y,\\
&&I_3:=4\int_{0}^{1}(X^{1/2})_y\left[K\Phi(\theta)\left(\frac{\delta}{X^{1/2}}-X^{1/2}\right)\right]_{y}\,\mathrm{d}y.
\end{eqnarray*}
In what follows, we shall repetitively argue by integration by parts, employing the boundary conditions $X_y(t,0)=X_y(t,1)=0$ as in Eq.~\eqref{IBVP}.

For $I_1$, we compute that
\begin{align*}
I_1 &= -8\int_{0}^{1}\left[(X^{1/2})_{yy}\right]^2\frac{D}{v^2}\,\mathrm{d}y + 8\int_{0}^{1}(X^{1/2})_y\left(\frac{D}{v^2}\right)_y\frac{(X^{1/2})_y^2}{X^{1/2}}\,\mathrm{d}y\\
&\qquad  +8\int_{0}^{1}(X^{1/2})_y\frac{D}{v^2}\left[\frac{2(X^{1/2})_y(X^{1/2})_{yy}}{X^{1/2}}-\frac{(X^{1/2})_y^3}{X}\right]\,\mathrm{d}y\\
&= -8\int_{0}^{1}X\frac{D}{v^2}\left[(\log X^{1/2})_{yy}\right]^2\,\mathrm{d}y + 8\int_{0}^{1}(X^{1/2})_y\left(\frac{D}{v^2}\right)_y\frac{(X^{1/2})_y^2}{X^{1/2}}\,\mathrm{d}y\\
&=-2\int_{0}^{1}X\frac{D}{v^2}\left[(\log X)_{yy}\right]^2\,\mathrm{d}y  + \int_0^1\left(\frac{D}{v^2}\right)_y \frac{\left(X_y\right)^3}{X^2} \,\dd y.
\end{align*}
Hence,
\begin{align}\label{concA}
I_1+I_2 &= \underbrace{ - 2\int_{0}^{1}\frac{D}{v^2} X\left[(\log X)_{yy}\right]^2\,\mathrm{d}y}_{=:\,[{\rm GOOD}]}  \underbrace{-2\int_{0}^{1}\left(\frac{D}{v^2}\right)_yX_y(\log X)_{yy}\,\mathrm{d}y}_{=:\,[{\rm BAD}]}.
\end{align}

The $[{\rm GOOD}]$ term is of favourable sign. For the $[{\rm BAD}]$ term, one may resort to Young's inequality to deduce, for arbitrarily small $\eta>0$, that
\begin{align}\label{intermediate}
\big|[{\rm BAD}]\big| &=\left| \int_0^1 \frac{D^{1/2}X^{1/2}}{v} \left[\log (X^{1/2})\right]_{yy} \cdot \frac{X_y}{X^{1/2}} \frac{(D/v^2)_y}{D^{1/2}\slash v}\,\dd y \right| \nonumber \\
&\leq \eta  \int_{0}^{1}\frac{D}{v^2}X\left[\left(\log X\right)_{yy}\right]^2\,\mathrm{d}y + C_{\eta}D \int_0^1 \frac{(X_y)^2}{X} \frac{(v_y)^2}{v^4}\,\dd y.
\end{align}
Here $C_\eta \sim \eta^{-1}$ is a positive constant. For the second term on the right-hand side of Eq.~\eqref{intermediate}, we have $v \geq m_0 >0$ by Proposition~\ref{th1} and $|v_y| \leq \Lambda$ by Theorem~\ref{thm: estimate on vy}, both uniformly in spacetime. Thus, recalling Notation~\ref{notation}, one may continue Eq.~\eqref{intermediate} by 
\begin{equation}\label{concB}
\big|[{\rm BAD}]\big| \leq \eta  \int_{0}^{1}\frac{D}{v^2}X\left[\left(\log X\right)_{yy}\right]^2\,\mathrm{d}y + C_{\eta} \Lambda \int_0^1 \frac{(X_y)^2}{X} \,\dd y.
\end{equation}

Now let us turn to $I_3$. Integration by parts yields that
\begin{align*}
I_3 = 4\int_{0}^{1}(X^{1/2})_{yy}K\Phi(\theta)\left[X^{1/2}-\frac{\delta}{X^{1/2}}\right]\,\mathrm{d}y.
\end{align*}
Thus, for any $\eta'>0$ we have that
\begin{align*}
\left|I_3\right| &\leq \eta' \int_0^1 \left[(X^{1/2})_{yy} \right]^2\,\dd y + C_{\eta'} K^2\int_0^1 \Phi^2(\theta)\left[X^{1/2}-\frac{\delta}{X^{1/2}}\right]^2\,\mathrm{d}y\\
&\leq \frac{13 \eta'}{8} \int_0^1 X \left[\left(\log X\right)_{yy}\right]^2\,\dd y + C_{\eta'} \Lambda \int_0^1 \left[X^{1/2}-\frac{\delta}{X^{1/2}}\right]^2\,\mathrm{d}y\\
&\leq \frac{13\eta' (M_0)^2}{8D} \int_{0}^{1}\frac{D}{v^2}X\left[\left(\log X\right)_{yy}\right]^2\,\mathrm{d}y+ C_{\eta'} \Lambda \int_0^1 \left[X^{1/2}-\frac{\delta}{X^{1/2}}\right]^2\,\mathrm{d}y,
\end{align*}
where $C_{\eta'}\sim (\eta')^{-1}$. In the above, we used the Young's inequality, Lemma~\ref{diffineq}, and the upper bounds for $v$ and $\theta$ in Proposition~\ref{th1}. Recall Notation~\ref{notation} for the constant $\Lambda$.

To proceed, simply notice that since $0\leq Z\leq 1$ (see Proposition~\ref{th1}; $Z \equiv Z^\e$ is the $C^1$-regularised variable as before), one has $\delta\leq X\leq 1+\delta$. This implies that $0 \leq 1-\frac{\delta}{X} \leq \frac{1}{1+\delta} <1$, and hence we have a bound uniform in $\delta$: 
\begin{align*}
0 \leq \left[X^{1/2}-\frac{\delta}{X^{1/2}}\right]^2 \leq {1+\delta} \leq 1.01
\end{align*}
whenever $\delta \in \left[0, 100^{-1}\right]$. Therefore, choosing $\eta'$ via
\begin{align*}
\eta =  \frac{13\eta' (M_0)^2}{8D}
\end{align*}
where $\eta$ is as in Eq.~\eqref{concB}, we arrive at the bound
\begin{align}\label{concC}
|I_3| \leq \eta  \int_{0}^{1}\frac{D}{v^2}X\left[\left(\log X\right)_{yy}\right]^2\,\mathrm{d}y + C_\eta\Lambda.
\end{align}
Here $\eta>0$ is arbitrary; the constants $C_\eta\sim\eta^{-1}$ and $\Lambda$ are both independent of $\delta$ and $\e$. 

Now, substituting the estimates~\eqref{concA}, \eqref{concB}, and \eqref{concC} into Eq.~\eqref{I1, I2, I3}, we deduce that
\begin{align*}
&\frac{\dd}{\dd t} \int_0^1\frac{X_y^2}{X}\,\mathrm{d}y + 2(1-\eta)\int_{0}^{1}\frac{D}{v^2}X\left[(\log X)_{yy}\right]^2\,\mathrm{d}y \leq  C_\eta \Lambda \left\{1+\int_0^1\frac{X_y^2}{X}\,\mathrm{d}y \right\}.
\end{align*} 
Fix any $0<\eta<1$ and hence $C_{\eta}$. Since $\Lambda$ is independent of both the ``shift'' and regularisation parameters $\delta$ and $\e$, we can now conclude the proof --- \emph{i.e.}, recover the $L^\infty_t L^2_y$-bound for $\frac{Z_y}{\sqrt{Z}}$ with $Z$ non-regularised --- from Gr\"{o}nwall's lemma.  \end{proof}

\section{Gradient estimate for reactant fraction on unbounded domains}\label{sec: gradient estimate on ubd}

In this section we prove Theorem~\ref{thm: main} for the unbounded domain case. For this purpose, we need the analogue of Theorem~\ref{thm: estimate on vy}, namely the uniform Lipschitz bound for $v$, on $\Omega = \R$. To avoid unnecessary repetitions with the bounded domain case, the parallel arguments are safely omitted. We shall only highlight the essential differences from the developments in \S\ref{sec: gradient estimate on bd}.

\begin{theorem}\label{thm: unbounded domain, vy bound}
Theorem~\ref{thm: estimate on vy} remains valid when $\Omega = \R$. That is, for any given $T>0$, we have $\left\|v_y\right\|_{L^\infty(Q_T)} \leq \Lambda$. The constant $\Lambda$ is as in Notation~\ref{notation}.
\end{theorem}

The proof is based on arguments by Jiang in \cite{j1}, which relies on the localisation lemma \emph{\`{a} la} Kazhikhov \cite{ks2}. See also Li \cite{li}.

\begin{proof}[Proof of Theorem~\ref{thm: unbounded domain, vy bound}]
   Throughout the proof we denote $I_k:=[k, k+1]$ for each $k \in \mathbb{Z}$.

   Consider the cutoff function $\chi=\chi_k\in W^{1,\infty}(\mathbb{R})$ defined by
   $$\chi(y)=\left\{\begin{aligned}
   	&1\qquad \text{for } y\leq k,\\
   	&k+1-y\qquad\text{for }  k\leq y\leq k+1,\\
   	&0\qquad\text{for }  y\geq k+1.
   \end{aligned}
   \right.$$
Multiply the equation $$u_t+\left(\frac{a\theta}{v}\right)_y=\left(\frac{\mu u_y}{v}\right)_y$$ by $\chi$ to obtain that
   $$\chi(y)u_t(t,y)=\left[\left(\frac{\mu u_y}{v}-\frac{a\theta}{v}\right)(t,y)\chi(y)\right]_y-\chi'(y)\left(\frac{\mu u_y}{v}-\frac{a\theta}{v}\right)(t,y).$$
 Then, integrating the above equality over $[0,t]\times ]y,\infty[$ and taking exponential, we deduce that
	$$\begin{aligned}
			v(t,y)\exp\left\{-\frac{a}{\mu}\int_{0}^t\frac{\theta}{v}(\tau,y)\mathrm{d}\tau\right\} = v_0(y)P(t)Q(t,y),
		\end{aligned}$$
where
	$$P(t):=\exp\left\{\frac{1}{\mu}\int_{0}^{t}\int_{k}^{k+1}\left(\frac{\mu u_y}{v}-\frac{a\theta}{v}\right)(\tau,\xi)\mathrm{d}\tau\mathrm{d}\xi\right\},$$
	$$Q(t,y):=v_0(y)\exp\left\{\frac{1}{\mu}\int_{y}^{\infty} [u_0(\xi)-u(t,\xi)]\chi(\xi)\mathrm{d}\xi\right\}.$$

One may infer the following from the above identity:
	\begin{align*}
	v(t,y)=Q(t,y)P(t)+\frac{a}{\mu}\int_0^t\frac{Q(t,y)P(t)}{Q(\tau,y)P(\tau)}\theta(\tau,y)\,\mathrm{d}\tau\qquad\text{for any } y\in I_k,\,t\geq 0.
	\end{align*}
Moreover, the bounds below are established by Jiang \cite{j1}:
	\begin{equation}
		\label{P}
		0\leq P(t)\leq \Lambda\exp\left\{-t/\Lambda\right\}\quad\text{ and }\quad \frac{P(t)}{P(\tau)}\leq \Lambda\exp\left\{-(t-\tau)/\Lambda\right\}.
	\end{equation}

Now, straightforward computation yields that
\begin{align}\label{vy, unbounded, rep formula}
v_y(t,y) &= Q_y(t,y)P(t)+\frac{a}{\mu}\int_0^t\frac{P(t)}{P(\tau)}\frac{Q_y(t,y)\theta(\tau,y)+Q(t,y)\theta_y(\tau,y)}{Q(\tau,y)}\,\mathrm{d}\tau\nonumber\\
&\qquad -\frac{a}{\mu}\int_0^t\frac{P(t)}{P(\tau)}\frac{Q(t,y)\theta(\tau,y)Q_y(\tau,y)}{Q^2(\tau,y)}\,\mathrm{d}\tau,
\end{align}
where
\begin{align*}
Q_y(t,y) &= v_{0y}(y)\exp\left\{\frac{1}{\mu}\int_{y}^{\infty} [u_0(\xi)-u(t,\xi)]\chi(\xi)\,\mathrm{d}\xi\right\}\\
&\qquad +v_0(y)\left[\frac{1}{\mu}[u(t,y)-u_0(y)](k+1-y)\right]\exp\left\{\frac{1}{\mu}\int_{y}^{\infty} [u_0(\xi)-u(t,\xi)]\chi(\xi)\,\mathrm{d}\xi\right\}.
\end{align*}
The above expressions are well defined, for it is assumed that $v_0 \in H^1(\R)\cap W^{1,\infty}(\R)$, as well as that $u(t,\bullet) \in H^1(\R)$ for every $t\in[0,T]$.

We are ready to conclude. Indeed, Eq.~\eqref{P} gives uniform bounds on $P(t)$ and $P(t)\slash P(\tau)$. Also, in view of the expressions for $Q$ and $Q_y$ above, the $L^\infty_{t,x}$-bound for $(u,\theta)$ and the $L^2_t H^2_x$-bound (hence the $L^2_t C^{0,1/2}_x$-bound by Sobolev--Morrey embedding) for $\theta$ --- proved in Proposition~\ref{th1} --- as well as the assumptions on $v_0$, also imply that $|Q_y| \leq \Lambda$ and $0<\Lambda^{-1} \leq Q \leq \Lambda < \infty$ uniformly in spacetime. The constants $\Lambda$ follow the same rule of Notation~\ref{notation}. Most importantly, all these above bounds  are uniform in $k$. Substituting these bounds into the representation formula~\eqref{vy, unbounded, rep formula} for $v_y$, we  complete the proof.    
\end{proof}

\begin{proof}[Sketch of proof of Theorem~\ref{thm: main} for unbounded domain]

The arguments in \S\ref{sec: gradient estimate on bd} almost carry over verbatim, with the understanding of the boundary condition for $Z$ as $\lim\limits_{|y|\rightarrow \infty} Z_y(t,y)=0$.  Indeed, consider $X=Z^\e+\delta$ as in \S\ref{sec: gradient estimate on bd}. The following holds for the case $\Omega = \R$ in place of Eq.~\eqref{I1, I2, I3}:
\begin{align*}
\frac{\mathrm{d}}{\mathrm{d}t}\int_{-\infty}^{\infty}\frac{X_y^2}{X}\,\mathrm{d}y &= 8\left(\sqrt{X}\right)_{yy}\left(\sqrt{X}\right)_y \frac{D}{v^2}\bigg|_{-\infty}^{\infty}-4\left(\sqrt{X}\right)_yK\Phi(\theta)\sqrt{X}\bigg|_{-\infty}^{\infty}+8\left(\sqrt{X}\right)_y^2\left(\frac{D}{v^2}\right)_y\bigg|_{-\infty}^{\infty}\\
		&-2\int_{-\infty}^{\infty}X\frac{D}{v^2}\left[\left(\log X\right)_{yy}\right]^2\mathrm{d}y-2\int_{-\infty}^{\infty}\left(\frac{D}{v^2}\right)_yX_y\left(\log X\right)_{yy}\,\mathrm{d}y\\
		&+4\int_{-\infty}^{\infty}\left(\sqrt{X}\right)_{yy}K\Phi(\theta)\sqrt{X}\,\mathrm{d}y.
\end{align*}
This identity holds pointwise in $t$. By virtue of the boundary conditions, the first three terms on the right-hand side are equal to zero. The control for the remaining terms is exactly the same as for the case $\Omega = [0,1]$, which relies essentially on Theorems~\ref{thm: unbounded domain, vy bound} and Lemma~\ref{th1}. \end{proof} 

\bigskip

\noindent
{\bf Acknowledgement}.
The research of SL is supported by NSFC (National Natural Science Foundation of China) Project $\#$12201399, the SJTU-UCL joint seed fund WH610160507/067, and the Shanghai Frontier Research Institute for Modern Analysis. This paper has been written during SL's visiting scholarship at New York University-Shanghai; SL is grateful to NYUSH for providing very nice working atmosphere.

\end{document}